\documentclass[11pt]{article}
\usepackage[a4paper,margin=0.9in]{geometry}
\usepackage[T1]{fontenc}
\usepackage{lmodern}
\usepackage{amsmath,amsthm,amssymb,mathtools}
\usepackage{microtype}
\usepackage{needspace}
\usepackage{hyperref}
\hypersetup{hidelinks,pdftitle={Critical Restricted Sumsets at the Boundary |A|+|B|=p},
  pdfauthor={Hongjian Li; Yangcheng Li; Pingzhi Yuan},
  pdfsubject={Complete boundary classification, arithmetic structure, and exact enumeration},
  pdfkeywords={restricted sumsets, inverse additive problems, prime fields, affine classification}}

\newtheorem{theorem}{Theorem}[section]
\newtheorem{lemma}[theorem]{Lemma}
\newtheorem{proposition}[theorem]{Proposition}
\newtheorem{corollary}[theorem]{Corollary}
\theoremstyle{definition}
\newtheorem{example}[theorem]{Example}

\newcommand{\Fp}{\mathbb F_p}
\newcommand{\RS}{\mathbin{\widehat{+}}}
\newcommand{\symdiff}{\mathbin{\triangle}}
\numberwithin{equation}{section}
\makeatletter
\renewcommand{\@maketitle}{%
  \begin{center}
    {\LARGE\@title\par}%
    \ifx\@author\@empty\else\vskip1em{\large\@author\par}\fi
    \ifx\@date\@empty\else\vskip0.6em{\@date\par}\fi
  \end{center}\par\vskip0.6em
}
\makeatother

\title{Critical Restricted Sumsets\\at the Boundary $\lvert A\rvert+\lvert B\rvert=p$}
\author{%
  Hongjian Li$^{1,}$\thanks{E-mail: \href{mailto:lhj@gdufs.edu.cn}{\texttt{lhj@gdufs.edu.cn}}.
    Supported by the Project of Guangdong University of Foreign Studies
    (Grant No.~2024RC063).}\quad
  Yangcheng Li$^{2,}$\thanks{Corresponding author.
    E-mail: \href{mailto:liyc@m.scnu.edu.cn}{\texttt{liyc@m.scnu.edu.cn}}.}\quad
  Pingzhi Yuan$^{2,}$\thanks{E-mail: \href{mailto:yuanpz@scnu.edu.cn}{\texttt{yuanpz@scnu.edu.cn}}.
    Supported by NSF of China Nos.~12571003 and~12501006, and by the
    Basic and Applied Basic Research Foundation of Guangdong Province
    No.~2024A1515010589.}\\[0.5em]
  {\small\itshape $^{1}$School of Mathematics and Statistics, Guangdong University of Foreign Studies,}\\
  {\small\itshape Guangzhou 510006, Guangdong, P. R. China}\\
  {\small\itshape $^{2}$School of Mathematical Sciences, South China Normal University,}\\
  {\small\itshape Guangzhou 510631, Guangdong, P. R. China}%
}
\date{}

\begin{document}
\maketitle

\begin{abstract}
Let $p$ be an odd prime, and write
$A\RS B=\{a+b:a\in A,\ b\in B,\ a\ne b\}$.
We classify all pairs $A,B\subseteq\Fp$ satisfying
$|A|+|B|=p$, $|A|>|B|=\ell\ge1$, and $|A\RS B|=p-2$.
After normalizing the two missing sums to $\{0,1\}$, we obtain exactly
$\ell+1$ explicit models. The classification forces $B\subseteq A$ and
yields, for fixed $p,\ell$, exactly $\binom p2(\ell+1)$ ordered pairs and
$1+\lfloor\ell/2\rfloor$ equivalence classes under simultaneous affine
transformations. We determine when either set is an arithmetic progression
and compute exact difference-set cardinalities. For $\ell\ge3$ and
$p\ge4\ell-5$, the cardinality $|B-B|$ determines the equivalence class at
fixed $p,\ell$. The proof uses punctured translates and cyclic component
counting. We also exhibit a critical pair in $\mathbb F_{13}$, with size
gap three and $|A|+|B|=p-2$, that contradicts a proposed inverse statement
below the boundary.
\end{abstract}
\noindent\textbf{Keywords:} restricted sumsets; inverse additive problems; prime fields; affine classification.

\setlength{\abovedisplayskip}{8pt plus 2pt minus 2pt}
\setlength{\belowdisplayskip}{8pt plus 2pt minus 2pt}
\setlength{\abovedisplayshortskip}{3pt plus 1pt}
\setlength{\belowdisplayshortskip}{6pt plus 2pt minus 2pt}

\section{Introduction and the boundary theorem}

Let $p$ be an odd prime. For nonempty subsets $A,B\subseteq\Fp$, write
$A+B=\{a+b:a\in A,\ b\in B\}$ and
$A\RS B=\{a+b:a\in A,\ b\in B,\ a\ne b\}$.
The restriction excludes equal summands as field elements.
For $|A|\ne|B|$, the Alon--Nathanson--Ruzsa theorem
\cite{ANR95,ANR96} gives
\begin{equation}\label{eq:ANR}
 |A\RS B|\ge\min\{p,\,|A|+|B|-2\}.
\end{equation}
A pair satisfying equality is called \emph{critical}.
The corresponding inverse problem asks for the structure of all such pairs.
For comparison, in the symmetric problem K\'arolyi \cite{Karolyi05} proved
that if $X\subseteq\Fp$, $|X|\ge5$, and $2|X|-3<p$, then
$|X\RS X|=2|X|-3$ holds exactly when $X$ is an arithmetic progression.

For distinct set sizes, Liu and Qian \cite[Theorem~1.5]{LQ26} proved that,
if at least one set is a progression, the conditions
\[
 |B|=\ell\ge3,\qquad |A|=k\ge\ell+3,\qquad k+\ell\le p
\]
force a critical pair to consist of progressions with a common difference,
with $B$ an initial or terminal $\ell$-term segment of $A$.
Their Conjecture~1.7 \cite{LQ26} asserts that,
under the same size conditions but without the progression hypothesis,
$|A\RS B|=k+\ell-2$ holds if and only if this endpoint description holds.
Li and Liang \cite[Example~2.1, version~1]{LL26} gave a counterexample in
$\mathbb F_{11}$ at $k+\ell=p$.

Our main problem is to classify every critical pair on the boundary
$|A|+|B|=p$, without assuming that either set is an arithmetic progression
or that $B\subseteq A$. The classification includes $|B|=1,2$ and adjacent
set sizes. For $\ell\ge4$, Liu and Qian \cite[Proposition~3.4]{LQ26}
already obtain the adjacent-size deletion description assuming
$B\subseteq A$; here containment follows from the full classification.

The boundary is distinguished by $|\Fp\setminus A|=|B|$:
each of the two missing restricted sums forces a translate of $-B$,
with at most one point removed, into a complement of the same size.
After normalization, these inclusions force $-B$ to have at most two
cyclic components in the standard directed cycle; equal cardinalities
then determine the complement of $A$. Thus the boundary counterexample
belongs to an explicit family, not an isolated exception.

The elementary proof of Theorem~\ref{thm:boundary} in
Section~\ref{sec:proof} also yields the boundary lower bound independently
of \eqref{eq:ANR}. Section~\ref{sec:structure} derives containment and all
progression types, counts pairs and affine classes, and computes exact
difference-set sizes. In particular, Proposition~\ref{prop:diff} shows
that $|B-B|$ determines the affine class when $\ell\ge3$ and $p\ge4\ell-5$.
Section~\ref{sec:below} gives a size-gap-three counterexample to
\cite[Theorem~3.1, version~1]{LL26} and compares it with known small-gap
constructions, organized by defect. These comparisons do not assert a
classification below the boundary.

Since $p$ is odd, two sizes summing to $p$ differ. Interchanging the sets
if necessary, the boundary problem reduces to
\begin{equation}\label{eq:range}
 |B|=\ell\ge1,\qquad |A|=p-\ell>\ell.
\end{equation}

For integers $0\le a\le b\le p-1$, write
$[a,b]=\{a,a+1,\ldots,b\}\subseteq\Fp$; an interval with $a>b$ is empty.
An \emph{arithmetic progression} is a set $a+r[0,m-1]$ with
$r\in\Fp^\times$ and $1\le m<p$; the progression may wrap around the field.
Having a \emph{common difference} means admitting such representations
with the same $r$, with either orientation allowed.
Two ordered pairs $(A,B)$ and $(A',B')$ are \emph{affinely equivalent} if
$A'=uA+c$ and $B'=uB+c$ for some $u\in\Fp^\times$ and $c\in\Fp$.
Thus the same affine transformation is applied to both sets.
Independent affine changes need not preserve the condition $a\ne b$.

\Needspace{7\baselineskip}
Set
\begin{equation}\label{eq:hd}
 h=\frac{p+1}{2},\qquad d=h-\ell.
\end{equation}
Here $h,d$ are integers; the residue of $h$ is $2^{-1}$ in $\Fp$.
Condition \eqref{eq:range} says precisely that $1\le\ell\le h-1$, so $d\ge1$.
The case $d=1$ is exactly $p=2\ell+1$.

\begin{theorem}\label{thm:boundary}
Assume \eqref{eq:range} and $|A\RS B|=p-2$.
Fix an ordering $(m_0,m_1)$ of the two missing sums and put
$\phi(x)=(m_1-m_0)^{-1}(x-m_0/2)$.
For $0\le v\le\ell$, define
\begin{equation}\label{eq:normalform}
 \begin{aligned}
 S_v&=[0,v-1]\cup[v+d,h-1],\\
 C_v&=[1,v]\cup[v+d,h-1].
 \end{aligned}
\end{equation}
Put $A_v=\Fp\setminus C_v$ and $B_v=-S_v$.
Then $(\phi(A),\phi(B))=(A_v,B_v)$ for a unique
$v\in\{0,1,\ldots,\ell\}$.
Conversely, every pair $(A_v,B_v)$ has missing restricted-sum set exactly
$\{0,1\}$. Thus there are exactly $\ell+1$ distinct pairs with this fixed
missing set.
\end{theorem}

The endpoint parameters are $v=0$, giving $S_0=C_0=[d,h-1]$, and
$v=\ell$, giving $S_\ell=[0,\ell-1]$, $C_\ell=[1,\ell]$.
For $1\le v\le\ell-1$, the set $S_v$ has two cyclic components; the two
intervals in $C_v$ join when $d=1$. Reversing the ordering of the missing
sums replaces $v$ by $\ell-v$, as shown in Corollary~\ref{cor:count}.
The classification proof works uniformly throughout \eqref{eq:range}.
For $p=11$ and $\ell=3$, the example in
\cite[Example~2.1, version~1]{LL26} is carried to the $v=1$ model by
$x\mapsto7x$. Thus that example belongs to the family classified here.

\section{Proof of the boundary classification}\label{sec:proof}

Throughout this section, assume \eqref{eq:range} and put
$S=-B$, $C=\Fp\setminus A$, so $|S|=|C|=\ell$.
A missing restricted sum is called \emph{ordinary} if it is absent from
$A+B$, and \emph{diagonal} otherwise.

\begin{lemma}\label{lem:puncture}
For $t\in\Fp$,
\begin{equation}\label{eq:puncture}
 t\notin A\RS B
 \quad\Longleftrightarrow\quad
 (t+S)\setminus\{t/2\}\subseteq C.
\end{equation}
If $t$ is missing, exactly one of the following occurs:
either $C=t+S$, in which case the hole is ordinary, or
\begin{equation}\label{eq:diagonal}
 C=\bigl((t+S)\setminus\{t/2\}\bigr)\cup\{x_t\},
 \qquad x_t\notin t+S,
\end{equation}
where $x_t$ is unique and $t/2\in A\cap B$; in this case the hole is diagonal.
\end{lemma}

\begin{proof}
A representation $t=a+b$ with $b\in B$ has $a=t-b$.
Its only forbidden possibility is $a=b=t/2$, proving
\eqref{eq:puncture}. Put $U=t+S$. Since $|U|=|C|=\ell$, there are two
possibilities. If $t/2\notin U\setminus C$, then
$U\setminus\{t/2\}\subseteq C$ together with either $t/2\notin U$ or
$t/2\in C$ gives $U\subseteq C$; equal cardinalities therefore imply
$C=U=t+S$. If instead $t/2\in U\setminus C$, then
$|U\setminus\{t/2\}|=\ell-1$. Hence $C$ contains those $\ell-1$ points
and, because $|C|=\ell$, exactly one further point $x_t$; necessarily
$x_t\notin U$, which gives \eqref{eq:diagonal} and also proves the
uniqueness of $x_t$.

Finally, $t/2\in U=t-B$ is equivalent to $t/2\in B$, while
$t/2\notin C$ means $t/2\in A$. Thus in the first case
$A\cap(t-B)=\varnothing$, so the hole is ordinary. In the second case
$t/2\in A\cap B$, and the only ordinary-sum representation not already
excluded by \eqref{eq:puncture} is the diagonal representation
$t=t/2+t/2$.
\end{proof}

\Needspace{10\baselineskip}
\begin{lemma}\label{lem:component}
Let $\varnothing\ne X\subsetneq\Fp$, and let $c(X)$ be the number of its
maximal intervals in the directed cycle $0,1,\ldots,p-1,0$. Then
\[
 c(X)=|X\setminus(X+1)|=|(X+1)\setminus X|
     =\tfrac12|X\symdiff(X+1)|.
\]
In particular, $c(X)=1$ if and only if $X$ is a single cyclic interval.
\end{lemma}

\begin{proof}
For $x\in\Fp$,
\[
 x\in X\setminus(X+1)
 \quad\Longleftrightarrow\quad
 x\in X\ \text{ and }\ x-1\notin X,
\]
so these are exactly the starting points of the maximal cyclic intervals.
Likewise,
\[
 x\in (X+1)\setminus X
 \quad\Longleftrightarrow\quad
 x-1\in X\ \text{ and }\ x\notin X,
\]
so these are exactly the immediate successors of their terminal points.
Each component contributes one point to each set. The two sets are disjoint,
and their union is $X\symdiff(X+1)$, giving the stated equalities.
\end{proof}

\Needspace{10\baselineskip}
\begin{proof}[Proof of Theorem~\ref{thm:boundary}]
\emph{Normalization and necessity.}
For $\phi(x)=ux+c$, $u\ne0$, one has
\begin{equation}\label{eq:affine}
 \phi(A)\RS\phi(B)=u(A\RS B)+2c,
\end{equation}
because $a\ne b$ is equivalent to $\phi(a)\ne\phi(b)$.
If the missing sums are $m_0,m_1$, take
$u=(m_1-m_0)^{-1}$ and $c=-um_0/2$.
The induced map on sums is $t\mapsto(t-m_0)/(m_1-m_0)$, so the
missing set becomes $\{0,1\}$. Replace $A,B$ by their images and retain
$S=-B$, $C=\Fp\setminus A$ for the normalized sets. The necessity argument
that follows uses only that $0$ and $1$ are missing, not that these are the
only missing sums.

By Lemma~\ref{lem:puncture}, the two holes give the alternatives
\begin{align*}
 (O_0):\;&C=S,
 & (D_0):\;&S\setminus C=\{0\},\quad C\setminus S=\{\alpha\};\\
 (O_1):\;&C=S+1,
 & (D_1):\;&(S+1)\setminus C=\{h\},\quad C\setminus(S+1)=\{\beta\}.
\end{align*}
The case $(O_0,O_1)$ would give $S=S+1$, impossible for a nonempty proper
subset of the prime cyclic group.

If $(D_0,O_1)$ holds, then $C=S+1$ and
$S\setminus(S+1)=\{0\}$. Hence Lemma~\ref{lem:component} gives
$c(S)=1$, and the unique cyclic component starts at $0$. Since its size is
$\ell$, it is exactly $[0,\ell-1]$; therefore $S=[0,\ell-1]$ and
$C=[1,\ell]$, giving \eqref{eq:normalform} with $v=\ell$.
If $(O_0,D_1)$ holds, then $C=S$ and $(S+1)\setminus S=\{h\}$.
Again $c(S)=1$, now with unique successor $h$ to the terminal point of the
component. Thus the component ends at $h-1$, and its size $\ell$ gives
$S=C=[h-\ell,h-1]$, which is \eqref{eq:normalform} with $v=0$.

It remains to consider $(D_0,D_1)$. The two diagonal conditions imply
\begin{equation}\label{eq:endpoints}
 0,h-1\in S,\qquad -1,h\notin S.
\end{equation}
Indeed, $D_0$ gives $0\in S$ and $0\notin C$; $D_1$ gives
$h-1\in S$ and $h\notin C$. If $h\in S$, then $D_0$ would put $h$ in $C$.
If $-1\in S$, then $D_1$ would put $0$ in $C$.
Furthermore,
\[
 |S\symdiff(S+1)|
 \le |S\symdiff C|+|C\symdiff(S+1)|=4,
\]
so Lemma~\ref{lem:component} gives at most two cyclic components.
Let $K_0$ and $K_h$ be the components containing $0$ and $h-1$,
respectively. They are distinct: if one component contained both points,
then, because $-1,h\notin S$, it would have to be the entire directed
interval $[0,h-1]$, whose size is $h>\ell$. Thus these are the only two
components of $S$. Since $-1,h\notin S$, both are contained in $[0,h-1]$;
moreover, $K_0$ starts at $0$ and $K_h$ ends at $h-1$. Consequently
\[
 S=[0,v-1]\cup[w,h-1]
\]
for some integers $0<v<w<h$.
Counting gives $v+h-w=\ell$, so $w=v+d$.
The two components are nonempty exactly when $1\le v\le\ell-1$;
in particular, this case cannot occur for $\ell=1$.

The two diagonal alternatives now determine $C$. From $(D_0)$ and
Lemma~\ref{lem:puncture} with $t=0$ we have
$S\setminus\{0\}\subseteq C$, while $(D_1)$ with $t=1$ and $1/2=h$ gives
$(S+1)\setminus\{h\}\subseteq C$. Thus both
\[
 \begin{aligned}
 U_0&=S\setminus\{0\}=[1,v-1]\cup[w,h-1],\\
 U_1&=(S+1)\setminus\{h\}=[1,v]\cup[w+1,h-1]
 \end{aligned}
\]
lie in $C$. Their union is $[1,v]\cup[w,h-1]$, whose size is
$v+h-w=\ell=|C|$. Since $U_0\cup U_1\subseteq C$ and the two sets have the
same cardinality, equality follows. This proves \eqref{eq:normalform} for
$1\le v\le\ell-1$. The calculation includes $v=1$ and $v=\ell-1$ under
the empty-interval convention; when $d=1$, the two intervals in $C$ are
adjacent but disjoint.

\emph{Sufficiency and uniqueness.}
In every displayed model, the two block sizes are $v$ and $\ell-v$.
Since $d\ge1$, the blocks do not overlap, so $|S_v|=|C_v|=\ell$.
Moreover,
\[
 S_v\setminus C_v\subseteq\{0\},\qquad
 (S_v+1)\setminus C_v\subseteq\{h\}.
\]
Thus $0,1\notin A_v\RS B_v$ by \eqref{eq:puncture}.
We verify directly that every other residue occurs. If $0<v<\ell$, then
$0,h\in B_v$ and $[h,p-1]\subseteq A_v$. Using standard integer
representatives, the admissible representations are
\[
 t=t+0\quad(h\le t\le p-1),\qquad
 t=(t+h-1)+h\pmod p\quad(2\le t\le h-1).
\]
In the second expression, $t+h-1\in[h+1,p-1]$ and $2h-1=p$, so the
summands are unequal and the sum is indeed $t$.

At the endpoints, put $k=p-\ell$, $I=[0,k-1]$, and $J=[0,\ell-1]$.
Reading the complements in \eqref{eq:normalform} gives
\begin{equation}\label{eq:endpoint-pairs}
 (A_0,B_0)=(h+I,h+J),\qquad (A_\ell,B_\ell)=(-I,-J).
\end{equation}
We have $I\RS J=[1,p-2]$: for $1\le t\le k-1$ use $t=t+0$;
for $k\le t\le p-2$ use $t=(k-1)+(t-k+1)$, where
$1\le t-k+1\le\ell-1<k-1$. The second range is empty if $\ell=1$.
The only possible sum $0$ is diagonal, and no integer sum exceeds $p-2$.
Equation~\eqref{eq:affine} now gives missing set $\{0,1\}$ for both pairs
in \eqref{eq:endpoint-pairs}.
Finally, $v$ is the first integer in $0,1,\ldots,\ell$ absent from $S_v$,
because $d\ge1$. This proves uniqueness in the fixed coordinates and
distinguishes all $\ell+1$ models.
\end{proof}

\begin{corollary}\label{cor:elementary}
Under \eqref{eq:range}, $|A\RS B|\ge p-2$.
\end{corollary}
\begin{proof}
Suppose, toward a contradiction, that three distinct residues
$m_0,m_1,m_2$ are missing. Normalize $m_0,m_1$ as in the proof of
Theorem~\ref{thm:boundary}. The induced map on sums is the bijection
\[
 t\longmapsto \frac{t-m_0}{m_1-m_0},
\]
so the image
\[
 \lambda=\frac{m_2-m_0}{m_1-m_0}
\]
is still missing and, by injectivity, satisfies $\lambda\notin\{0,1\}$.
The necessity argument above uses only that $0$ and $1$ are missing, so it
forces one of the normalized models. The sufficiency part then says that
this model has missing set exactly $\{0,1\}$, contradicting the missing
residue $\lambda$. Hence at most two residues are missing, and therefore
$|A\RS B|\ge p-2$. This argument does not use \eqref{eq:ANR}.
\end{proof}

\Needspace{10\baselineskip}
\section{Arithmetic structure and exact enumeration}\label{sec:structure}

We first record the integer lifting observation used to distinguish
progressions from the interior models.

\begin{lemma}\label{lem:lifting}
Let $X\subseteq[0,(p-1)/2]$ be an arithmetic progression in $\Fp$,
with $|X|\ge3$. Then its standard integer representatives form an
arithmetic progression in $\mathbb Z$. If $X$ contains two consecutive
integers, then $X$ is an integer interval.
\end{lemma}

\begin{proof}
List the standard integer representatives $x_0,\ldots,x_{m-1}$ in
progression order. Then, for $0\le i\le m-3$,
\[
 x_{i+2}-2x_{i+1}+x_i\equiv0\pmod p,\qquad
 |x_{i+2}-2x_{i+1}+x_i|\le p-1.
\]
The only multiple of $p$ in this range is $0$, so every second difference
vanishes over the integers. Thus the representatives form an integer
progression with a nonzero integer difference $R$.
If $x_j-x_i=\pm1$, then $(j-i)R=\pm1$. Since $j-i$ and $R$ are nonzero
integers, $|j-i|=|R|=1$. Hence $R=\pm1$ and $X$ is an integer interval.
\end{proof}

\begin{proposition}\label{prop:arithmetic}
Under \eqref{eq:range}, the models of Theorem~\ref{thm:boundary} have the
following properties.
\begin{enumerate}
\renewcommand{\labelenumi}{(\roman{enumi})}
\setlength{\itemsep}{2pt}
\item Every model satisfies $B_v\subseteq A_v$, and
\[
 |A_v+B_v|=
 \begin{cases}
 p-1,&v=0,\ell,\\
 p,&1\le v\le\ell-1.
 \end{cases}
\]
The endpoints are precisely the models admitting progression representations
with a common nonzero difference; then $B_v$ is an initial or terminal
segment of $A_v$.
\item When $p=2\ell+1$, the normal forms simplify to
\begin{equation}\label{eq:deletion}
 C_v=[1,\ell],\qquad A_v=-[0,\ell],\qquad
 B_v=A_v\setminus\{-v\}.
\end{equation}
Consequently, the boundary critical pairs with $|A|=|B|+1$ are exactly a progression
of length $\ell+1$ together with the set obtained by deleting any one term.
\item For $\ell=1$ there are no interior parameters. For $\ell=2$, the sole
interior model has both sets progressions, but with no common difference.
For $\ell\ge3$ and $1\le v\le\ell-1$, the set $B_v$ is not a progression,
and $A_v$ is a progression if and only if $p=2\ell+1$.
\end{enumerate}
\end{proposition}

\begin{proof}
Since $C_v\subseteq[1,h-1]$ and $B_v\subseteq\{0\}\cup[h,p-1]$,
the sets $B_v,C_v$ are disjoint, giving $B_v\subseteq A_v$.
For the ordinary sumset, recall that Theorem~\ref{thm:boundary} gives
\[
 \Fp\setminus(A_v\RS B_v)=\{0,1\}.
\]
If a restricted-missing residue $t$ nevertheless belongs to $A_v+B_v$,
then every ordinary representation must be diagonal; equivalently,
$t/2\in A_v\cap B_v$. Conversely, such a diagonal point gives an ordinary
representation. Now $0\notin C_v$, so $0\in A_v$, and
$0\in B_v$ exactly when $0\in S_v$, namely when $v>0$. Hence $0$ belongs
to $A_v+B_v$ exactly for $v>0$. Likewise $h=1/2$ never lies in $C_v$, so
$h\in A_v$, while
\[
 h\in B_v\quad\Longleftrightarrow\quad -h=h-1\in S_v
 \quad\Longleftrightarrow\quad v<\ell.
\]
Thus $1=h+h$ belongs to $A_v+B_v$ exactly for $v<\ell$. Every other
residue is already in the restricted sumset. Therefore the ordinary missing
set is $\{0\}$ at $v=0$, $\{1\}$ at $v=\ell$, and empty in the interior,
which proves the ordinary-sum formula.

The endpoint representations \eqref{eq:endpoint-pairs} give the asserted
common difference and segment description. Negation preserves progressions,
as does complementation: if $X=a+r[0,m-1]$, then
$\Fp\setminus X=a+r[m,p-1]$.

For (ii), if $p=2\ell+1$, then $h=\ell+1$ and $d=1$, so
\eqref{eq:normalform} gives $C_v=[1,\ell]$ and
$S_v=[0,\ell]\setminus\{v\}$. This proves \eqref{eq:deletion} for every
$\ell\ge1$. Conversely, write any length-$(\ell+1)$ progression as
$P=a+r[0,\ell]$. If its $v$th term is deleted, the affine map
$x\mapsto a-rx$ sends $(A_v,B_v)$ in \eqref{eq:deletion} to
$(P,P\setminus\{a+rv\})$. Criticality follows from
Theorem~\ref{thm:boundary} and \eqref{eq:affine}.

For (iii), the assertion for $\ell=1$ is immediate. More explicitly, if
$B=\{b\}$ and $|A|=p-1$, then $A\RS B=(A\setminus\{b\})+b$,
so criticality is exactly $b\in A$.
For $\ell=2$, the sets $B_v,C_v$ have two elements; hence $A_v,B_v$ are
progressions. At $v=1$, multiplication by $2$ gives
$2\cdot B_1=\{0,1\}$ and $2\cdot C_1=\{2,-1\}$.
A common progression difference for the scaled pair would be $\pm1$.
By complementation, this would also be a difference for $\{2,-1\}$,
whose two elements instead differ by $\pm3$.
Since $3\not\equiv\pm1\pmod p$ for $p\ge5$, no common difference exists.

Now let $\ell\ge3$ and $1\le v\le\ell-1$.
The set $S_v\subseteq[0,h-1]$ has two nonempty blocks separated by a gap
of length $d\ge1$. Their total size is $\ell\ge3$, so at least one block
contains two consecutive integers. Were $S_v$ a progression,
Lemma~\ref{lem:lifting} would force it to be an integer interval,
contradicting the gap. Thus $S_v$, and hence $B_v=-S_v$, is not a progression. When
$d\ge2$, the same argument applies to $C_v$, whose two blocks are separated
by a nonempty gap of length $d-1$; therefore $C_v$ is not a progression,
and neither is its complement $A_v$. When $d=1$, part~(ii) already shows
that $A_v$ is a progression. These interior cases also exclude a common
difference, completing (i) and (iii).
\end{proof}

\Needspace{10\baselineskip}
\begin{corollary}\label{cor:count}
For fixed $p$ and $1\le\ell\le(p-1)/2$, let $N(p,\ell)$ count critical
ordered pairs with $|A|=p-\ell$ and $|B|=\ell$. Then
\begin{equation}\label{eq:count}
 N(p,\ell)=\binom p2(\ell+1).
\end{equation}
Exactly $p(p-1)$ of these pairs admit progression representations with a
common difference; the remaining $\binom p2(\ell-1)$ pairs have the interior
types described in Proposition~\ref{prop:arithmetic}.
There are exactly $1+\lfloor\ell/2\rfloor$ affine equivalence classes,
represented by \eqref{eq:normalform} with $0\le v\le\lfloor\ell/2\rfloor$:
one endpoint class and $\lfloor\ell/2\rfloor$ interior classes.
\end{corollary}

\begin{proof}
For each unordered missing set $M=\{m_0,m_1\}$, fix one ordering and use
\eqref{eq:affine} to normalize it to $\{0,1\}$.
This is a bijection from pairs with missing set $M$ to the $\ell+1$ pairs
of Theorem~\ref{thm:boundary}. The ordering is fixed for each $M$, not
counted as additional data. Distinct missing sets give disjoint collections
of pairs, proving \eqref{eq:count}. The two endpoint models give
$2\binom p2=p(p-1)$ common-difference pairs, and
Proposition~\ref{prop:arithmetic} gives the other types.

Any affine equivalence between normalized models must preserve their
missing set $\{0,1\}$. Its induced map $t\mapsto ut+2c$ is therefore either
the identity or $t\mapsto1-t$, so the map on elements is the identity or
$\rho(x)=h-x$. Under $\rho$, the complement transforms as $C\mapsto h-C$,
whereas $S=-B$ transforms as $S\mapsto-h-S=h-1-S$. Thus
\begin{equation}\label{eq:reflection}
 (S_v,C_v)\longmapsto(h-1-S_v,\ h-C_v)
                  =(S_{\ell-v},C_{\ell-v}).
\end{equation}
Thus normalized models are identified precisely by $v\leftrightarrow\ell-v$.
The endpoints form one orbit and the interior parameters form
$\lfloor\ell/2\rfloor$ orbits. When $\ell$ is even, the middle model
$v=\ell/2$ is fixed by $\rho$. This proves the affine class count throughout
the full range, including $\ell=1,2$ and $p=2\ell+1$.
\end{proof}

\Needspace{10\baselineskip}
\begin{proposition}\label{prop:diff}
Assume \eqref{eq:range}. For $\ell\ge2$ and $1\le v\le\ell-1$,
put $m=\max\{v,\ell-v\}$. Then
\begin{equation}\label{eq:diff-sizes}
 \begin{aligned}
 |B_v-B_v|&=2\ell-1+2\min\{m-1,d\},\\
 |C_v-C_v|&=2\ell-1+2\min\{m-1,d-1\}.
 \end{aligned}
\end{equation}
For $v=0,\ell$, both difference sets instead have size $2\ell-1$.
Moreover, for fixed $p,\ell$ with $\ell\ge3$ and $p\ge4\ell-5$, two
boundary critical pairs are affinely equivalent if and only if
their smaller sets have difference sets of the same cardinality.
\end{proposition}

\begin{proof}
Let $W=\{-(m-1),\ldots,m-1\}\subseteq\Fp$. The two blocks of $S_v$ have
lengths $v$ and $\ell-v$, so their within-block difference sets are the
centered windows of radii $v-1$ and $\ell-v-1$; their union is therefore
$W$, where $m=\max\{v,\ell-v\}$. The same observation applies to $C_v$.

We compute the cross differences using the following identity for
nonempty integer intervals, before reduction modulo $p$:
\[
 [a,b]-[c,e]=[a-e,b-c].
\]
All cross-difference intervals below are given in standard representatives.
For $S_v$, the positive differences from the second block to the first are
\[
 [v+d,h-1]-[0,v-1]=[d+1,h-1].
\]
Their negatives have standard representatives $[h,p-d-1]$.
For $C_v$, the positive cross differences are
\[
 [v+d,h-1]-[1,v]=[d,h-2],
\]
and their negatives are $[h+1,p-d]$. Thus
\begin{equation}\label{eq:diff-unions}
 \begin{aligned}
 S_v-S_v&=W\cup[d+1,p-d-1],\\
 C_v-C_v&=W\cup[d,h-2]\cup[h+1,p-d].
 \end{aligned}
\end{equation}
Each cross union has $2\ell-2$ elements.

It remains only to count the part of $W$ not already contained in the
cross differences. In standard representatives,
\[
 W=[0,m-1]\cup[p-m+1,p-1].
\]
Put $q=\min\{m-1,d\}$. Then, with the empty-interval convention,
\[
 W\setminus[d+1,p-d-1]=[0,q]\cup[p-q,p-1],
\]
so this complement has $1+2q$ elements. Therefore
\[
 |S_v-S_v|=(2\ell-2)+(1+2q)
 =2\ell-1+2\min\{m-1,d\}.
\]
For $C_v$, put $q'=\min\{m-1,d-1\}$. The central gap
$\{h-1,h\}$ between the two cross intervals does not meet $W$, because
$m\le\ell-1\le h-2$. Hence
\[
 W\setminus\bigl([d,h-2]\cup[h+1,p-d]\bigr)
   =[0,q']\cup[p-q',p-1],
\]
which has $1+2q'$ elements. This gives
\[
 |C_v-C_v|=2\ell-1+2\min\{m-1,d-1\}.
\]
Since negation preserves difference-set cardinality,
$|B_v-B_v|=|S_v-S_v|$, proving \eqref{eq:diff-sizes}.

For either endpoint, $S_v$ and $C_v$ are intervals of length $\ell$.
Their difference sets are the $2\ell-1$ consecutive residues
$-(\ell-1),\ldots,\ell-1$, which are distinct modulo $p$ because
$2\ell-1\le p$.

Affine transformations preserve difference-set cardinalities.
If $p\ge4\ell-5$, then $d\ge\ell-2\ge m-1$, and the interior formula becomes
$|B_v-B_v|=2\ell+2m-3$. For $\ell\ge3$ this is larger than the endpoint
value and recovers $m=(|B_v-B_v|-2\ell+3)/2$. The value of $m$ determines
$\{v,\ell-v\}$, hence exactly the affine class described in
Corollary~\ref{cor:count}. This proves the final assertion.
\end{proof}

\Needspace{8\baselineskip}
\section{Below the boundary: contrasting mechanisms}\label{sec:below}

The equality $|\Fp\setminus A|=|B|$ is essential to the boundary
argument. For $|A|>|B|\ge1$ and $|A|+|B|\le p$, put
$\delta=p-|A|-|B|$. If $\delta>0$, the complement has size
$|B|+\delta$, so the punctured-translate inclusions no longer determine it
by equality of cardinalities.
A critical pair has $\delta+2$ missing restricted sums. To distinguish
ordinary and diagonal losses, put
$\nu(A,B)=|(A+B)\setminus(A\RS B)|$; thus
\[
 p-|A+B|+\nu(A,B)=\delta+2.
\]
The examples below compare the roles of the defect $\delta$ and the size
gap $|A|-|B|$, beginning with a counterexample satisfying the
size-gap-three hypotheses.

\paragraph{A size-gap-three counterexample.}
Theorem~3.1 of version~1 of \cite{LL26} asserts that critical pairs with
$|B|\ge3$, $|A|\ge|B|+3$, and $|A|+|B|\le p-1$ consist of progressions
with a common difference, the smaller being an initial or terminal segment
of the larger. The following example satisfies every size hypothesis but
has neither set a progression.

\begin{example}\label{ex:13}
In $\mathbb F_{13}$, let
\[
 B=\{0,1,3,9\},\qquad A=\{0,1,3,7,8,9,11\}.
\]
The restricted sums are enumerated by fixing $b\in B$:
\[
\begin{array}{c|l}
 b & (A\setminus\{b\})+b\\ \hline
 0 & \{1,3,7,8,9,11\}\\
 1 & \{1,4,8,9,10,12\}\\
 3 & \{1,3,4,10,11,12\}\\
 9 & \{3,4,7,9,10,12\}.
\end{array}
\]
Taking the union gives
\[
 A\RS B=\mathbb F_{13}\setminus\{0,2,5,6\},\qquad
 |A\RS B|=9=|A|+|B|-2.
\]
Neither set is an arithmetic progression. Indeed, direct counting gives
\[
\begin{array}{c|rrrrrr}
 r&1&2&3&4&5&6\\ \hline
 |A\cap(A+r)|&3&4&3&3&4&4\\
 |B\cap(B+r)|&1&1&1&1&1&1
\end{array}
\]
and the counts for $-r$ are the same. An $m$-term progression of
nonzero difference $r$ has exactly $m-1$ elements in common with its
translate by $r$, whereas the counts above are less than $6$ for $A$
and less than $3$ for $B$.
Since $|A|=7=|B|+3$ and $|A|+|B|=11<12=p-1$, this critical pair
contradicts \cite[Theorem~3.1, version~1]{LL26}.

The second row of counts also shows that $B$ is a cyclic $(13,4,1)$
difference set. All four missing restricted sums reappear as
$2b=b+b$, $b\in B\subseteq A$, since $2\cdot B=\{0,2,5,6\}$.
Thus $A+B=\mathbb F_{13}$, $\nu(A,B)=4$, and $\delta=2$.
\end{example}

\paragraph{Comparison with small-gap families.}
The one- and two-deletion patterns below occur in Liu and Qian
\cite[Lemma~2.4(i), Proposition~3.2(i)]{LQ26}; the latter proposition is
stated for $\ell\ge4$. We record explicit sumsets for $\ell\ge3$ and
organize these constructions by the defect $\delta$.

\Needspace{10\baselineskip}
\begin{proposition}\label{prop:holes}
Let $\ell\ge3$.
\begin{enumerate}
\renewcommand{\labelenumi}{(\roman{enumi})}
\setlength{\itemsep}{2pt}
\item If $p\ge2\ell+1$ and $1\le j\le\ell-1$, set
$A=[0,\ell]$, $B=A\setminus\{j\}$. Then
\[
 A\RS B=[1,2\ell-1],\qquad A+B=[0,2\ell].
\]
This is a critical pair with size gap one and even defect
$\delta=p-2\ell-1$.
\item If $p\ge2\ell+3$, set
$A=[0,\ell+1]$, $B=[0,\ell]\setminus\{\ell-1\}$. Then
\[
 A\RS B=[1,2\ell-1]\cup\{2\ell+1\},\qquad A+B=[0,2\ell+1].
\]
This is a critical pair with size gap two and odd defect
$\delta=p-2\ell-2$.
\end{enumerate}
In both cases $A$ is a progression, $B$ is not a progression, and
$(A+B)\setminus(A\RS B)=\{0,2\ell\}$, so $\nu(A,B)=2$.
Consequently, every fixed defect $\delta\ge0$ admits such mixed critical
pairs over every sufficiently large prime field.
\end{proposition}

\begin{proof}
All integer sums displayed below are less than $p$, so there is no modular
wrap-around. In (i), use $z=z+0$ for $1\le z\le\ell$ and
$z=(z-\ell)+\ell$ for $\ell+1\le z\le2\ell-1$.
The endpoints $0,\ell$ belong to $B$, so these representations are
admissible regardless of $j$. The sums $0,2\ell$ have only their diagonal
representations. These facts give both asserted sumsets.

In (ii), the sums with $b=0$ cover $[1,\ell+1]$, and those with $b=\ell$
cover $[\ell+1,2\ell-1]\cup\{2\ell+1\}$ non-diagonally.
The only candidates for $2\ell$ are $\ell+\ell$ and
$(\ell+1)+(\ell-1)$; the first is diagonal and the second uses a point
absent from $B$. The sum $0$ is also only diagonal.
The ordinary sums with $b=0,\ell$ fill $[0,2\ell+1]$.

In either case $B\subseteq[0,(p-1)/2]$, contains consecutive integers,
and is not an integer interval. Lemma~\ref{lem:lifting} therefore shows
that $B$ is not a progression.
For any prescribed even $\delta\ge0$, choose an odd prime
$p\ge\delta+7$ and put $\ell=(p-\delta-1)/2$ in (i).
For odd $\delta\ge1$, use any odd prime $p\ge\delta+8$ and
$\ell=(p-\delta-2)/2$ in (ii).
\end{proof}

At $\delta=0$, part~(i) recovers the interior deletion construction of
Proposition~\ref{prop:arithmetic}(ii), up to negation; positive $\delta$
extends it below the boundary. Both families have size gap less than three.

\Needspace{22\baselineskip}
\begin{example}\label{ex:small}
Let $p\ge7$ and $a,b\in\Fp^\times$ with $a\ne b$ and $a\ne-b$.
Set $A=\{0,a,b,a+b\}$ and delete any one element to obtain $B$.
Then
\[
 A\RS B=\{a,b,a+b,2a+b,a+2b\},\qquad |A\RS B|=5.
\]
Indeed, the six unordered pairs of distinct elements of $A$ give precisely
these sums, with $a+b$ occurring twice. Deleting one vertex leaves at least
one endpoint of each such pair in $B$, so no restricted sum is lost.
Any difference of two of the five displayed values lies in
$\{\pm a,\pm b,\pm(a-b),\pm(a+b),\pm2a,\pm2b\}$ and is nonzero.
Thus the pair is critical with size gap one and defect $\delta=p-7$.
The example $A=\{0,1,3,4\}$, $B=\{0,1,3\}$ in $\mathbb F_{11}$ is
already given in \cite[Remark~3.5]{LQ26}. More generally, for every $p\ge11$,
the same choice
\[
 A=\{0,1,3,4\},\qquad B=\{0,1,3\}
\]
has neither set a progression by Lemma~\ref{lem:lifting}: both lie in
$[0,(p-1)/2]$, contain consecutive integers, and are not intervals.
Here $A\RS B=\{1,3,4,5,7\}$ and
$A+B=[0,7]$, so $\nu(A,B)=3$.
\end{example}

At the boundary, normalization yields the parameter $v$, with affine
equivalence given by $v\leftrightarrow\ell-v$. Below it, ordinary and
diagonal losses vary. Neither set need be a progression: the four-point
examples have size gap one, whereas Example~\ref{ex:13} meets the
size-gap-three hypotheses. We make no uniqueness assertion for the
$\mathbb F_{13}$ example and no classification below the boundary.

\end{document}